\documentclass[11pt,a4paper]{amsart}
\usepackage{graphics,epic}
\usepackage{amsmath,amssymb, amsthm}
\usepackage[all,2cell]{xy}
\usepackage{times}
\usepackage{mathrsfs}
\usepackage[top=25mm,bottom=25mm,left=30mm,right=30mm]{geometry}

\newtheorem{theorem}{Theorem}[section]
\newtheorem*{theorem*}{Theorem}

\newtheorem{lemma}[theorem]{Lemma}
\newtheorem{proposition}[theorem]{Proposition}

\newtheorem*{conjecture*}{Conjecture}

\newcommand{\ie}{{\em i.e.}\ }

\newcommand{\opname}[1]{\operatorname{\mathsf{#1}}}

\renewcommand{\mod}{\opname{mod}\nolimits}

\newcommand{\add}{\opname{add}\nolimits}

\newcommand{\der}{\cd}

\newcommand{\Z}{\mathbb{Z}}

\renewcommand{\P}{\mathbb{P}}

\newcommand{\Hom}{\opname{Hom}}
\newcommand{\go}{\opname{G_0}}

\newcommand{\Ext}{\opname{Ext}}

\newcommand{\End}{\opname{End}}

\newcommand{\cd}{{\mathcal D}}

\renewcommand{\hat}[1]{\widehat{#1}}

\begin{document}

\title[Reduction approach to silting theory over hereditary categories]{A reduction approach to silting objects for derived categories of hereditary categories}\thanks{Partially supported by the National Natural Science Foundation of China (Grant No. 11971326)}

%\sffamily
\author[Dai]{Wei Dai}
\address{Wei Dai\\ Department of Mathematics\\SiChuan University\\610064 Chengdu\\P.R.China}
\email{375670160@qq.com}

\author[Fu]{Changjian Fu}
\address{Changjian Fu\\Department of Mathematics\\SiChuan University\\610064 Chengdu\\P.R.China}
\email{changjianfu@scu.edu.cn}

\begin{abstract}
Let $\mathcal{H}$ be a hereditary abelian category over a field $k$ with finite dimensional $\Hom$ and $\Ext$ spaces. It is proved that the bounded derived category $\der^b(\mathcal{H})$ has a silting object   iff $\mathcal{H}$ has a tilting object iff $\der^b(\mathcal{H})$ has a simple-minded collection with acyclic $\Ext$-quiver. Along the way, we obtain a new proof for the fact that every presilting object of $\der^b(\mathcal{H})$ is a partial silting object. We also consider the question of complements for pre-simple-minded collections. In contrast to presilting objects, a pre-simple-minded collection $\mathcal{R}$ of $\der^b(\mathcal{H})$ can be completed into a simple-minded  collection iff the $\Ext$-quiver of $\mathcal{R}$ is acyclic. 
\end{abstract}

\maketitle
\section{Introduction}
Throughout this note, let $k$ be a field. By a hereditary abelian category, we mean a hereditary abelian category over $k$ with finite dimensional $\Hom$ and $\Ext$ spaces.

Hereditary abelian categories with tilting objects and their bounded  derived categories provide a framework for the classical tilting theory, which were extensively studied since early eighties.The main examples of such categories are the category $\mod H$ of finitely generated right modules over a finite dimensional hereditary $k$-algebra $H$ and the category $\opname{coh}\mathbb{X}$ of coherent sheaves over an exceptional curve $\mathbb{X}$ in the sense of Lenzing~\cite{Lenzing}. A remarkable theorem of Happel and Retiten \cite{HR} shows that a connected hereditary abelian category with tilting object is either derived equivalent to $\mod H$ or to $\opname{coh}\mathbb{X}$.

Silting objects were first introduced in \cite{KV} as a generalization of tilting objects to parametrize bounded $t$-structures on derived categories of path algebras of Dynkin quivers.  Recent years, the topic has obtained a lot of attention due to the work of  Aihara and Iyama \cite{AiharaIyama}, in which a mutation theory for silting objects has been developed.  Moreover, a reduction theorem has been proved, which establishes a correspondence between certain silting objects in a triangulated category $\mathcal{T}$ and silting objects in its Verdier quotient $\mathcal{T}/\mathcal{S}$ with respect to a thick subcategory of $\mathcal{S}$.
Various connections between silting objects and other topics in representation theory have been discovered, such as bounded $t$-structures, co-$t$-structures, torsion pairs and simple-minded collections and so on (cf. \cite{KY, BY} for instance). 

The aim of this note is to study the bounded derived category 
 $\der^b(\mathcal{H})$ of a hereditary abelian category $\mathcal{H}$ from the viewpoint of silting theory.  
It is known that there are triangulated categories which do not admit a silting object. Our first result is a characterization of the existence of silting objects of $\der^b(\mathcal{H})$. 
\begin{theorem}\label{t:main-thm-1}
Let $\mathcal{H}$ be a hereditary abelian category. The following are equivalent:
\begin{enumerate}
\item $\mathcal{H}$ has a tilting object;
\item $\der^b(\mathcal{H})$ has a tilting object;
\item $\der^b(\mathcal{H})$ has a silting object;
\item $\der^b(\mathcal{H})$ has a simple-minded collection whose $\opname{Ext}$-quiver is acyclic.
\end{enumerate}
\end{theorem}
 The equivalence between $(1)$ and $(2)$ was proved in \cite[Theorem 1.7]{HR98} and our result yields a new proof for this fact.

 Let $\mathcal{T}$ be a Krull-Schmidt triangulated category with silting objects.
 One of open questions in silting theory is whether a presilting object in $\mathcal{T}$ can be completed into a silting object (cf. \cite[Question 3.13]{BY} and \cite[Question 2.14]{Aihara})?  The following result gives a positive answer for the bounded derived category of a hereditary abelian category with tilting objects.
 \begin{theorem}~\label{t:presilting-partial}
Let $\mathcal{H}$ be a hereditary abelian category with tilting objects. Every presilting object of $\der^b(\mathcal{H})$ can be completed into a silting object.
\end{theorem}
We remark that Theorem \ref{t:presilting-partial} is not new. In particular, Br\"{u}stle and Yang \cite{BY} have suggested a proof by  the transitivity of the action of braided group on exceptional sequences. In \cite{LL,XuYang}, the result has been proved for $\mathcal{H}=\mod H$ for a finite dimensional hereditary $k$-algebra $H$ by  different methods.

Simple-minded collection is a dual notion of silting object.  We consider the analogous question of complements for a pre-simple-minded collection in $\der^b(\mathcal{H})$. In contrast to presilting objects, there are pre-simple-minded collections which can not be completed into simple-minded collections.
\begin{theorem}\label{t:main-thm-3}
Let $\mathcal{H}$ be a hereditary abelian category with tilting objects. A pre-simple-minded collection $\mathcal{X}$ of $\der^b(\mathcal{H})$ can be completed into a simple-minded collection if and only if the $\opname{Ext}$-quiver  of $\mathcal{X}$ is acyclic.
\end{theorem}

Our proofs of Theorem \ref{t:main-thm-1}--\ref{t:main-thm-3} are inspired by the reduction  approach of \cite{FG}, where the Iyama-Yoshino's reduction was applied to study the connectedness of cluster-tilting graph of a hereditary abelian category. In present paper, we apply silting reduction to investigate silting objects in the bounded derived category $\der^b(\mathcal{H})$ of a hereditary abelian category $\mathcal{H}$. A key observation is that the localization of $\der^b(\mathcal{H})$ with respect to the thick subcategory generated by an exceptional object is triangle equivalent to the bounded derived category of another hereditary abelian category (cf. Lemma \ref{l:prep-cat} and Theorem \ref{t:verdier-quotient}).

The paper is organized as follows. In Section \ref{s:preliminary}, we  recall basic results for silting theory ans simple-minded collections. In Section \ref{s:hereditary-cat}, we investigate the  localization of the bounded derived category of a hereditary abelian category  with respect to an exceptional object. In particular, Theorem \ref{t:verdier-quotient} is proved. We present the proofs of Theorem \ref{t:main-thm-1}, Theorem \ref{t:presilting-partial} and Theorem \ref{t:main-thm-3} in Section \ref{s:proofs}.

\subsection*{Notation}
Let $\mathcal{T}$ be a triangulated category and $\mathcal{X}, \mathcal{Y}$ two full subcategories of $\mathcal{T}$.
\begin{itemize}
\item We always denote by $[1]$ the suspension functor of $\mathcal{T}$ unless otherwise stated.
\item Denote by $\mathcal{X}*\mathcal{Y}$ the subcategory of $\mathcal{T}$ consisting of objects $Z$ which admits a triangle $X\to Z\to Y\to X[1]$, where $X\in \mathcal{X}, Y\in \mathcal{Y}$.
\item For an integer $l$, set $\mathcal{X}[l]:=\{X[l]~|~\forall\ X\in \mathcal{X}\}$.
\item Let $\add \mathcal{X}$ be the smallest full subcategory of $\mathcal{T}$ which is closed under finite coproducts, summands, isomorphisms and containing $\mathcal{X}$. If $\mathcal{X}$ consists of a single object $X$, we simply denote it by $\add X$.
\item Denote by $\opname{thick}(\mathcal{X})$ the thick subcategory of $\mathcal{T}$ containing $\mathcal{X}$.
\item If $\mathcal{T}$ is Krull-Schmidt and $M\in \mathcal{T}$, denote by $|M|$ the number of pairwise non-isomorphic indecomposable direct summands of $M$.
\end{itemize}

\section{Preliminaries}  \label{s:preliminary}

\subsection{Perpendicular category and Verdier quotient} Let $\mathcal{T}$ be a Krull-Schmidt triangulated category and $\mathcal{M}$ a subcategory of $\mathcal{T}$.   A morphism $f: M\to N$ is a {\it right $\mathcal{M}$-approximation} of $N\in \mathcal{T}$ if $M\in \mathcal{M}$ and $\Hom_\mathcal{T}(M', f)$ is surjective for any $M'\in \mathcal{M}$.
The subcategory $\mathcal{M}\subset \mathcal{T}$ is {\it contravariantly finite} if every object in $\mathcal{T}$ has a right $\mathcal{M}$-approximation. Dually, we define a {\it left $\mathcal{M}$-approximation} and {\it covariantly finite subcategory}. We say that $\mathcal{M}$ is {\it functorially finite} if it is contravariantly finite and covariantly finite.

Define
\[\mathcal{M}^\perp:=\{N\in \mathcal{T}~|~\Hom_\mathcal{T}(M,N)=0~\text{for all $M\in \mathcal{M}$}\}
\]
and 
\[ \!^\perp\mathcal{M}:=\{N\in \mathcal{T}~|~\Hom_\mathcal{T}(N,M)=0~\text{for all $M\in \mathcal{M}$}\}.
\]
The subcategory $\mathcal{M}^\perp$ (resp. $ \!^\perp\mathcal{M}$) is called the {\it right} (resp. {\it left) perpendicular} category of $\mathcal{M}$ in $\mathcal{T}$. If $\mathcal{M}$ is a triangulated subcategory of $\mathcal{T}$, then both $\mathcal{M}^\perp$ and $ \!^\perp\mathcal{M}$ are triangulated subcategories of $\mathcal{T}$.

Recall that a pair of subcategories $(\mathcal{X}, \mathcal{Y})$ of $\mathcal{T}$ is a {\it torsion pair} of $\mathcal{T}$, if  $\Hom_{\mathcal{T}}(\mathcal{X}, \mathcal{Y})=0$ and $\mathcal{X}*\mathcal{Y}=\mathcal{T}$. The following useful result is known as Wakamatsu's Lemma (cf. \cite[Lemma 2.22]{AiharaIyama}).
\begin{lemma}\label{l:wakamatsu-lemma}
Let $\mathcal{M}$ be a subcategory of $\mathcal{T}$ such that $\mathcal{M}*\mathcal{M}\subseteq \mathcal{M}$.
\begin{itemize}
\item[(1)]  If $\mathcal{M}$ is contravariantly finite, then $(\mathcal{M}, \mathcal{M}^\perp)$ is a torsion pair of $\mathcal{T}$;
\item[(2)] If $\mathcal{M}$ is covariantly finite, then $(\!^\perp\mathcal{M}, \mathcal{M})$ is a torsion pair of $\mathcal{T}$.
\end{itemize}
\end{lemma}

 Let $\mathcal{S}$ be a thick subcategory of $\mathcal{T}$ and  $\mathcal{T}/\mathcal{S}$ the Verdier quotient of $\mathcal{T}$ with respect to $\mathcal{S}$.  Denote by $\mathbb{L}: \mathcal{T}\to \mathcal{T}/\mathcal{S}$ the localization functor. 
 We denote by $\iota:\mathcal{S}^\perp \hookrightarrow \mathcal{T}$(resp. $\iota: ^\perp\mathcal{S}\hookrightarrow \mathcal{T}$)  the inclusion functor. The following is well-known, which identifies the Verdier quotient $\mathcal{T}/\mathcal{S}$ with certain subcategories of $\mathcal{T}$.
\begin{lemma}\label{l:perp-quotient}
Let $\mathcal{S}$ be a thick subcategory of $\mathcal{T}$.
\begin{enumerate}
\item If $\mathcal{S}$ is contravariantly finite, then the composition $\mathbb{L}\circ \iota: \mathcal{S}^\perp\to \mathcal{T}/\mathcal{S}$ is an equivalence of triangulated categories;
\item If  $\mathcal{S}$ is covariantly finite, then the composition $\mathbb{L}\circ \iota:~  ^\perp\mathcal{S}\to \mathcal{T}/\mathcal{S}$ is an equivalence of triangulated categories.
\end{enumerate}
\end{lemma}
\begin{proof}
It is straightforward to check that the functor $\mathscr{L}:=\mathbb{L}\circ \iota$ induces an isomorphism \[\Hom_{\mathcal{T}}(X,Y)\cong \Hom_{\mathcal{T}/\mathcal{S}}(\mathscr{L}(X),\mathscr{L}(Y))\] for any $X,Y\in \mathcal{S}^\perp$ (resp. $^\perp \mathcal{S}$).
On the other hand, the functor $\mathscr{L}$ is dense by Lemma \ref{l:wakamatsu-lemma}.
\end{proof}

\subsection{Silting theory}\label{s:silting-theory}
We follow \cite{AiharaIyama}. For simplicity, we only consider $\Hom$-finite Krull-Schmidt triangulated categories and silting objects.

Let $\mathcal{T}$ be a $\Hom$-finite Krull-Schmidt triangulated category. An object $M$ of $\mathcal{T}$ is a {\it presilting} object if $\Hom_\mathcal{T}(M,M[i])=0$ for all $i>0$. A presilting object $M$ is {\it silting} if $\opname{thick}(M)=\mathcal{T}$.  A silting object $M$ of $\mathcal{T}$ is a {\it tilting object} if $\Hom_\mathcal{T}(M,M[i])=0$ for $i\neq 0$.
It is known that there exist $\Hom$-finite Krull-Schmidt triangulated categories which do not admit silting objects.

Let $T=M\oplus \overline{T}$ be a basic silting object of $\mathcal{T}$. Consider the triangle
\[N\to T_M\xrightarrow{f_M}M\to N[1],
\]
where $f_M$ is a minimal right $\add \overline{T}$-approximation of $M$. According to  \cite[Theorem 2.31]{AiharaIyama},  $N\oplus \overline{T}$ is a basic silting object of $\mathcal{T}$ and $N\oplus \overline{T}$ is called the {\it right mutation} of $T$ with respect to $M$. Dually, if we consider the triangle induced by a minimal left $\add \overline{T}$-approximation of $M$, we obtain the {\it left mutation} of $T$ with respect to $M$.

Let $\mathcal{T}$ be a $\Hom$-finite Krull-Schmidt triangulated category with a silting object $T$. It follows from \cite[Theorem 2.27]{AiharaIyama} that the Grothendieck group $\go(\mathcal{T})$ of $\mathcal{T}$ is a free abelian group of rank $|T|$. In particular, the images of the indecomposable direct summands of $T$ in $\go(\mathcal{T})$ form a $\Z$-basis of $\go(\mathcal{T})$. As a consequence, each silting object of $\mathcal{T}$ has the same number of pairwise non-isomorphic indecomposable direct summands. A presilting object $M\in \mathcal{T}$ is a {\it partial silting} object if there is an object $N\in \mathcal{T}$ such that $M\oplus N$ is a silting object.
It is an open question that whether a presilting object in $\mathcal{T}$ is a partial silting object (cf. \cite[Question 3.13]{BY})?

We denote by $\opname{silt} \mathcal{T}$ the set of isomorphism classes of basic silting objects of $\mathcal{T}$. 
The following reduction theorem plays a central role in our investigation.
\begin{theorem}\cite[Theorem 2.37]{AiharaIyama}\label{t:silting-reduction}
	Let $\mathcal{T}$ be a $\Hom$-finite Krull-Schmidt triangulated category, $\mathcal{S}$ a functorially finite thick subcategory of $\mathcal{T}$ and $\mathcal{T}/\mathcal{S}$ the Verdier quotient. Denote by $\mathbb{L}:\mathcal{T}\to \mathcal{T}/\mathcal{S}$ the localization functor. For any $D\in \opname{silt}\ \mathcal{S}$,  there is a bijective map
	\[\{T\in \opname{silt}\ \mathcal{T}~|~D\in \add T\}\to \opname{silt}\ \mathcal{T}/\mathcal{S}
	\] given by $T\mapsto \mathbb{L}(T) $.
\end{theorem}
Let us recall the inverse map of the bijection following the proof of \cite[Theorem 2.37]{AiharaIyama}.
Denote by $\mathcal{S}_D^{\leq 0}:=\cup_{l\geq 0}\add D*\add D[1]*\cdots*\add D[l]$ and $\mathcal{S}_{D}^{<0}:=\mathcal{S}_D^{\leq 0}[1]$. It is known that $\mathcal{S}_{D}^{<0}$ is covariantly finite in $\mathcal{T}$.  Since $\mathcal{S}$ is functorially finite, we may identify $ \mathcal{T}/\mathcal{S}$ with $\mathcal{S}^\perp$. Let $N\in \mathcal{S}^\perp$ be a silting object of $\mathcal{S}^\perp$.
Consider the following triangle
\[
S_N\to T_N\to N\xrightarrow{g}S_N[1],
\]
where $S_N[1]\in \mathcal{S}_{D}^{<0}$ and $g$ is a minimal left $\mathcal{S}_{D}^{<0}$-approximation. According to the proof of \cite[Theorem 2.37]{AiharaIyama},  $T_N\oplus D$ is a silting object of $\mathcal{T}$ such that $\mathbb{L}(T_N\oplus D)=\mathbb{L}(N)$.

\subsection{Simple-minded collections}\label{ss:smc}
Let $\mathcal{T}$ be a $\Hom$-finite $k$-linear triangulated category and $\mathcal{X}=\{X_1,\ldots, X_r\}$ a collection of objects. We call $\mathcal{X}$ a {\it pre-simple-minded collection(=pre-SMC)} if the following conditions hold for $i,j=1,\ldots, r$
\begin{enumerate}
\item[$\bullet$] $\Hom_\mathcal{T}(X_i,X_j[m])=0$ for any $m<0$;
\item[$\bullet$] $\End_{\mathcal{T}}(X_i)$ is a division algebra and $\Hom_\mathcal{T}(X_i,X_j)$ vanishes for $i\neq j$;
\end{enumerate}
In particular, every object in a pre-SMC is indecomposable.
For a pre-SMC $\mathcal{X}$, its {\it $\opname{Ext}$-quiver} $Q_\mathcal{X}$ is defined as follows.
\begin{itemize}
\item The vertices of $Q_\mathcal{X}$ are indexed by objects of $\mathcal{X}$;
\item  For $X_i,X_j\in \mathcal{X}$,  there are $\frac{\dim_k\Hom_\mathcal{T}(X_i,X_j[1])}{\dim_kEnd_\mathcal{T}(X_i)}$ arrows from $X_i$ to $X_j$.
\end{itemize}

A pre-SMC $\mathcal{X}$ of $\mathcal{T}$ is a {\it simple-minded collection}(=SMC)({\it cohomologically Schurian} in \cite{Al}) if  $\opname{thick}(\mathcal{X})=\mathcal{T}$.  
 Similar to the case of silting objects, if $\mathcal{T}$ admits a SMC $\mathcal{X}$, then the Grothendieck group $\go(\mathcal{T})$ of $\mathcal{T}$ is a free abelian group of rank $|\mathcal{X}|$. We denote by $\opname{SMC} \mathcal{T}$ the set of isomorphism classes of SMCs of $\mathcal{T}$.
 
 Let $\mathcal{R}$ be a pre-SMC of $\mathcal{T}$.  Denote by $\opname{SMC}_\mathcal{R}\mathcal{T}$ the set of  isomorphism classes of SMCs of $\mathcal{T}$ containing $\mathcal{R}$. Let $\mathcal{H}_\mathcal{R}$ be the smallest extension-closed subcategory of $\mathcal{T}$ containing $\mathcal{R}$. Define 
 \[\mathcal{Z}:=\mathcal{R}[\geq 0]^\perp\cap \ ^\perp\mathcal{R}[\leq 0].
 \]
The following reduction theorem for SMCs has been established in \cite{Jin}.
\begin{theorem}\label{t:smc-reduction}\cite[Theorem 3.1]{Jin}
Assume that $\mathcal{H}_\mathcal{R}$ satisfies the following conditions:
\begin{itemize}
\item $\mathcal{H}_\mathcal{R}$ is contravariantly finite in $\mathcal{R}[>0]^\perp$ and covariantly finite in $^\perp\mathcal{R}[<0]$;
\item For any $X\in \mathcal{T}$, we have $\Hom_\mathcal{T}(X,\mathcal{H}_\mathcal{R}[i])=0=\Hom_\mathcal{T}(\mathcal{H}_\mathcal{R},X[i])$ for $i\ll 0$.
\end{itemize}
Then
\begin{enumerate}
\item The composition $\mathcal{Z}\hookrightarrow \mathcal{T}\to \mathcal{T}/\opname{thick}(\mathcal{R})$ is an additive equivalence $\mathcal{Z}\xrightarrow{\sim}\mathcal{T}/\opname{thick}(\mathcal{R})$;
\item There is a bijection
\[\opname{SMC}_\mathcal{R}\mathcal{T}\to \opname{SMC} \mathcal{T}/\opname{thick}(\mathcal{R})
\]
sending $\mathcal{X}\in \opname{SMC}_\mathcal{R}\mathcal{T}$ to $\mathcal{X}\backslash \mathcal{R}\in \opname{SMC}\mathcal{T}/\opname{thick}(\mathcal{R})$.
\end{enumerate}
\end{theorem}
  We may regard $\mathcal{Z}$ as a triangulated category via the additive equivalence $\mathcal{Z}\xrightarrow{\sim}\mathcal{T}/\opname{thick}(\mathcal{R})$ in Theorem \ref{t:smc-reduction} (1). Denote by $\langle 1\rangle$ the suspension functor of $\mathcal{Z}$. Then for each object $Z\in \mathcal{Z}$, $Z\langle 1\rangle$ is determined by the following  triangle of $\mathcal{T}$
  \[R_Z\xrightarrow{f_Z} Z[1]\to Z\langle 1\rangle \to R_Z[1],
  \]
  where $f_Z$ is a minimal right $\mathcal{H}_\mathcal{R}$-approximation of $Z[1]$ (cf.  \cite[Lemma 3.4]{Jin}).

The inverse map of the bijection in Theorem \ref{t:smc-reduction} (2) is constructed as follows.
Let $\overline{\mathcal{X}}$ be a SMC of $\mathcal{T}/\opname{thick}(\mathcal{R})$. Denote by $\hat{\mathcal{X}}\subset \mathcal{Z}$ the preimage of $\overline{\mathcal{X}}$ via the equivalence in $(1)$. Then $\hat{\mathcal{X}}\cup\mathcal{R}$ is a SMC of $\mathcal{T}$, which is the preimage of $\overline{\mathcal{X}}$.

\section{Hereditary abelian categories with tilting objects}\label{s:hereditary-cat}
\subsection{Hereditary abelian categories}
Let $\mathcal{H}$ be a hereditary abelian category and $\der^b(\mathcal{H})$ the bounded derived category of $\mathcal{H}$.  Recall that an object $M\in \der^b(\mathcal{H})$ is {\it rigid} if $\Hom_{\der^b(\mathcal{H})}(M,M[1])=0$. It is {\it exceptional} if it is rigid and indecomposable.
The following fundamental result is due to Happel and Ringel~\cite{HappelRingel}.
\begin{lemma}\label{l:HR-exceptional}
Let $E$ and $F$ be indecomposable objects in $\mathcal{H}$ such that $\Hom_{\der^b(\mathcal{H})}(F, E[1])=0$. Then any nonzero homomorphism $f:E\to F$ is a monomorphism or epimorphism. In particular, the endomorphism ring of an exceptional object is a division algebra.
\end{lemma}

Let $\mathcal{M}$ be a full subcategory of $\der^b(\mathcal{H})$ and $M\in \mathcal{M}$ an indecomposable object. A path in $\mathcal{M}$ from $M$ to itself is a {\it cycle} in $\mathcal{M}$, that is a sequence of nonzero non-isomorphism between indecomposable objects in $\mathcal{M}$ of the form
\[M=M_0\xrightarrow{f_1}M_1\xrightarrow{f_2}M_2\to \cdots\xrightarrow{f_r}M_r=M.
\]
The following is a consequence of Lemma~\ref{l:HR-exceptional} (cf. \cite[Lemma 4.2]{Fu} or \cite[Corollary 4.2]{HappelRingel}).
\begin{lemma}\label{l:no-cycle}
Let $T$ be an object in $\der^b(\mathcal{H})$ such that $\Hom_{\der^b(\mathcal{H})}(T, T[1])=0$. Then the subcategory $\add T$ has no cycle.
\end{lemma}

\subsection{Hereditary categories with tilting objects}
Let $\mathcal{H}$ be a hereditary abelian category.  A rigid object $T\in \mathcal{H}$ is a {\it tilting object} provided that for $X\in \mathcal{H}$ with $\Hom_{\mathcal{H}}(T,X)=0=\Ext^1_{\mathcal{H}}(T,X)$, we have $X=0$.

Throughout this subsection, we always assume that $\mathcal{H}$ admits a tilting object.
As a consequence, the Grothendieck group $\go(\der^b(\mathcal{H}))$ is a free abelian group of finite rank. We denote by $\opname{rank}\go(\der^b(\mathcal{H}))$  the rank of $\go(\der^b(\mathcal{H}))$. If $T\in \mathcal{H}$ is a tilting object, then $\opname{rank}\go(\der^b(\mathcal{H}))=|T|$.
The existence of tilting objects also implies that $\der^b(\mathcal{H})$ admits almost split triangles and hence $\mathcal{H}$ has almost split sequences \cite{HRS}.  
Denote by $\tau: \der^b(\mathcal{H})\to \der^b(\mathcal{H})$ the Auslander-Reiten(=AR) translation functor, which restricts to the AR translation $\tau:\mathcal{H}\to \mathcal{H}$.
The following is a reformulation of \cite[Proposition 1.2(a)]{HR}.
\begin{proposition}~\label{p:no-proj-inj}
	Assume that $\mathcal{H}$ is connected.
	If $\mathcal{H}$ is not equivalent to  $\mod H$ for a finite dimensional hereditary $k$-algebra $H$, then $\mathcal{H}$ has neither nonzero projective objects nor nonzero injective objects. Consequently, the AR translation $\tau: \mathcal{H}\to \mathcal{H}$ is an equivalence.
\end{proposition}
\begin{proof}
	According to \cite[Proposition 1.2(a)]{HR}, $\mathcal{H}$ has no nonzero projective objects. Hence for each indecomposable $X\in \mathcal{H}$, we have $\tau X\in \mathcal{H}$. Let $X$ be an indecomposable object of $\mathcal{H}$. Let $X\to E\to \tau^{-1}X\xrightarrow{h} X[1]$ be the almost split triangle starting at $X$. To show that $X$ is not injective, it suffices to show that $\tau^{-1}X\in \mathcal{H}$. Since $h$ is nonzero, we conclude that $\tau^{-1}X\in \mathcal{H}$ or $\tau^{-1}X\in \mathcal{H}[1]$. Suppose that $\tau^{-1}X\in \mathcal{H}[1]$, then $\tau^{-1}X[-1]\in \mathcal{H}$. We have $X[-1]=\tau(\tau^{-1}X[-1])\in \mathcal{H}$, a contradiction. In particular, $\mathcal{H}$ has no nonzero injective objects.
\end{proof}

We also have the following result from \cite[Lemma 3.7]{Happel98}, where the proof is valid for arbitrary field.
\begin{lemma}\label{l:Happel-tilting}
	Let $\mathcal{H}$ be a hereditary category with tilting object. Let $M\in \mathcal{H}$ be a rigid object such that $|M|=\opname{rank}G_0(\der^b(\mathcal{H}))$, then $M$ is a tilting object.
\end{lemma}
A rigid object $M\in \mathcal{H}$ is a {\it partial tilting} if there is an object $N\in \mathcal{H}$ such that $M\oplus N$ is a tilting object.  The following seems to be known for experts, which has been proved by Happel and Unger \cite[Proposition 3.1]{HU} over algebraically closed field. Here we sketch a proof for arbitrary fields by cluster-tilting theory and we refer to \cite{Buanetal} for unexplained terminology in cluster-tilting theory.
\begin{proposition}~\label{p:rigid-partial}
	Each rigid object of $\mathcal{H}$ is a partial tilting object.
\end{proposition}
\begin{proof}
	Without loss of generality, we may assume that $\mathcal{H}$ is connected.
	Let $M$ be a rigid object of $\mathcal{H}$.
	If $\mathcal{H}=\mod H$ for a finite dimensional hereditary $k$-algebra $H$, it is well-known that $M$ is a partial tilting module by classical tilting theory. 
	
	Let us assume that $\mathcal{H}$ is not equivalent to $\mod H$ for a finite dimensional hereditary $k$-algebra $H$. By Proposition~\ref{p:no-proj-inj}, $\mathcal{H}$ has neither nonzero projective objects nor nonzero injective objects. Let $\mathcal{C}(\mathcal{H}):=\der^b(\mathcal{H})/ \tau^{-1}\circ [1]$ be the cluster category of $\mathcal{H}$, \ie  the orbit category of $\der^b(\mathcal{H})$ by the equivalent functor $\tau^{-1}\circ [1]$( {cf.} \cite{Keller}).  The cluster category $\mathcal{C}(\mathcal{H})$ admits a canonical triangle structure such that the projection functor $\pi: \der^b(\mathcal{H})\to \mathcal{C}(\mathcal{H})$ is a triangle functor. The projection functor $\pi$ induces a bijection between the set of isomorphism classes of objects of $\mathcal{H}$ and the set of isomorphism classes of objects of $\mathcal{C}(\mathcal{H})$. Moreover, an object $X\in \mathcal{H}$ is rigid if and only if $\pi(X)$ is rigid in $\mathcal{C}(\mathcal{H})$ (cf. \cite{Buanetal, Zhu}). 
	
	Let $T$ be a tilting object of $\mathcal{H}$, then $\pi(T)$ is a cluster-tilting object of $\mathcal{C}(\mathcal{H})$ ({cf.} \cite{Amiot}).  Since $M$ is rigid,  $\pi(M)$ is rigid, which can be completed to a cluster-tilting object of $\mathcal{C}(\mathcal{H})$ by \cite[Theorem 4.1]{AIR}, say $N\in \mathcal{H}$ such that $\pi(M)\oplus \pi(N)$ is a cluster-tilting object of $\mathcal{C}(\mathcal{H})$. As a consequence, $M\oplus N$ is rigid. According to \cite[Corollary 4.5]{AIR}, $|M\oplus N|=|\pi(M)\oplus \pi(N)|=|T|$. We conclude that $M\oplus N$ is a tilting object of $\mathcal{H}$ by Lemma~\ref{l:Happel-tilting}.
\end{proof}

\subsection{Perpendicular category and localization}

Let $\mathcal{H}$ be a hereditary abelian category and $E\in \mathcal{H}$ an exceptional object.
Define \[E^\perp:=\{X\in \mathcal{H}~|~\Hom_\mathcal{H}(E,X)=0=\Ext^1_{\mathcal{H}}(E,X)\}\]
and 
 \[\!^\perp E:=\{X\in \mathcal{H}~|~\Hom_\mathcal{H}(X,E)=0=\Ext^1_{\mathcal{H}}(X,E)\}.\] It is straightforward to check that $E^\perp$ and $^\perp E$ are  hereditary abelian subcategories of $\mathcal{H}$. 
 \begin{lemma}\label{l:prep-cat}
 Let $\mathcal{H}$ be a hereditary abelian category and $E\in \mathcal{H}$ an exceptional object. 
 We have $\der^b(E^\perp)\cong \der^b(\mathcal{H})/\opname{thick}(E) $ and $\der^b(^\perp E)\cong\der^b(\mathcal{H})/\opname{thick}(E)$.  
 \end{lemma}
 \begin{proof}
 Since $E$ is exceptional, according to Lemma~\ref{l:HR-exceptional}, the indecomposable objects of $\opname{thick}(E)$ are precisely $E[i], i\in \Z$. 
It follows that $\opname{thick}(E)$ is a functorially finite subcategory of $\der^b(\mathcal{H})$. By Lemma \ref{l:perp-quotient}, we have $\opname{thick}(E)^\perp\cong \der^b(\mathcal{H})/\opname{thick}(E)$. Consider the inclusion functor $\iota: E^\perp \hookrightarrow \mathcal{H}$, which induces a fully faithful triangle functor $\hat{\iota}: \der^b(E^\perp)\hookrightarrow \der^b(\mathcal{H})$. It is straightforward to check that the image of $\hat{\iota}$ coincides with $\opname{thick}(E)^\perp$. Consequently, $\der^b(E^\perp)\cong \opname{thick}(E)^\perp \cong \der^b(\mathcal{H})/\opname{thick}(E)$. Similarly, one can prove $\der^b(\!^\perp E)\cong \der^b(\mathcal{H})/\opname{thick}(E)$.
 
 \end{proof}

\begin{lemma}\label{l:silting-tilting}
Let $\mathcal{H}$ be a hereditary abelian category and $E\in \mathcal{H}$ an exceptional object. Assume that $E^\perp$ has a tilting object $M$. Then $M\oplus E[1]$ is a silting object of $\der^b(\mathcal{H})$. Moreover, the right mutation of $M\oplus E[1]$ with respect to $E[1]$ is a tilting object of $\mathcal{H}$.
\end{lemma}
\begin{proof}
Note that $M\in E^\perp$ is a tilting object and $\mathcal{H}$ is hereditary, we clearly have
\[\Hom_{\der^b(\mathcal{H})}(M\oplus E[1], M[i]\oplus E[i+1])=0\ \text{for all $i>0$.}
\]
Recall that $\opname{thick}(E)$ is a functorially finite subcategory of $\der^b(\mathcal{H})$.  For any $X\in \der^b(\mathcal{H})$, consider the following triangle
\[E_X\xrightarrow{f_X}X\to Z\to E_X[1],
\]
where $f_X$ is a minimal right $\opname{thick}(E)$-approximation of $X$. It follows that $Z\in \opname{thick}(E)^\perp$ by  Lemma \ref{l:wakamatsu-lemma}.  
As in the proof of Lemma \ref{l:prep-cat}, we may identify $\der^b(E^\perp)$ with $\opname{thick}(E)^\perp$.  By the assumption that $M$ is a tilting object of $E^\perp$, we conclude that $\opname{thick}(M)=\opname{thick}(E)^\perp$. Consequently, $X\in \opname{thick}(M\oplus E[1])$. Hence $\opname{thick}(M\oplus E[1])=\der^b(\mathcal{H})$. In particular, we have proved that $M\oplus E[1]$ is a silting object of $\der^b(\mathcal{H})$.

Let $M'\xrightarrow{f_{E[1]}} E[1]$ be a minimal right $\add M$-approximation of $E[1]$, which fits into the following triangle
\[
E\to N\to M'\xrightarrow{f_{E[1]}} E[1].
\]
It follows from \cite[Theorem 2.31]{AiharaIyama} that $M\oplus N$ is a silting object of $\der^b(\mathcal{H})$. Note that $\mathcal{H}$ is an extension-closed subcategory of $\der^b(\mathcal{H})$, we conclude that $N\in \mathcal{H}$. It is straightforward to show that $N\oplus M$ is a tilting object of $\mathcal{H}$. Indeed, since $N\oplus M$ is a silting object, we have $\Ext^1_{\mathcal{H}}(N\oplus M, N\oplus M)=0$ and $\opname{thick}(N\oplus M)=\der^b(\mathcal{H})$. Let $L\in \mathcal{H}$ such that $\Hom_\mathcal{H}(N\oplus M, L)=0=\Ext^1_\mathcal{H}(N\oplus M, L)$.  We clearly have $\Hom_{\der^b(\mathcal{H})}(N\oplus M, L[i])=0$ for all $i\in \mathbb{Z}$.
Consequently, for any object $Y\in \opname{thick}(N\oplus M)=\der^b(\mathcal{H})$, we have $\Hom_{\der^b(\mathcal{H})}(Y,L)=0$. In particular, $\Hom_{\mathcal{H}}(L,L)=0$, which implies that $L=0$. This completes the proof.
\end{proof}

 The following result plays a fundamental role in our reduction approach to silting objects in $\der^b(\mathcal{H})$.
\begin{theorem}\label{t:verdier-quotient}
Let $\mathcal{H}$ be a hereditary abelian category with tilting objects and  $E\in \mathcal{H}$ an exceptional object. 
There is a hereditary abelian category $\mathcal{H}'$ with tilting objects such that $\der^b(\mathcal{H}')$ is triangle equivalent to $\der^b(\mathcal{H})/\opname{thick}(E)$. Moreover, $\opname{rank} G_0(\der^b(\mathcal{H}))=\opname{rank} G_0(\der^b(\mathcal{H}'))+1$.
\end{theorem}
\begin{proof}
Without loss of generality, we may assume that $\mathcal{H}$ is connected. By \cite[Theorem 3.5]{HR}, $\mathcal{H}$ is either derived equivalent to the category $\mod H$ of finitely generated modules over a finite dimensional hereditary $k$-algebra $H$ or the category $\opname{coh}\mathbb{X}$ of coherent sheaves over an exceptional curve $\mathbb{X}$ in the sense of Lenzing~\cite{Lenzing}. Since we are working with derived categories, we may further assume that $\mathcal{H}=\mod H$ or $\mathcal{H}=\opname{coh}\mathbb{X}$.

Let us consider the case $\mathcal{H}=\opname{coh}\mathbb{X}$. By Proposition \ref{p:rigid-partial}, there is an object $M\in \mathcal{H}$ such that $X\oplus M$ is a tilting object of $\mathcal{H}$. Applying \cite[Proposition 1.4]{HR}, we conclude that $\mathcal{H}':=E^\perp$ is a connected  hereditary abelian category with tilting object.  By Lemma \ref{l:prep-cat}, we obtain that $\der^b(\mathcal{H}')\cong \der^b(\mathcal{H})/\opname{thick}(E)$.
Let $N$ be a basic tilting object of $\mathcal{H}'=E^\perp$, it follows from Lemma \ref{l:silting-tilting} that $N\oplus E[1]$ is a basic silting object of $\der^b(\mathcal{H})$. Consequently,
$\opname{rank} G_0(\der^b(\mathcal{H}))=\opname{rank} G_0(\der^b(\mathcal{H}'))+1$.

Now we turn to the case $\mathcal{H}=\mod H$. Without loss of generality, we may assume that $E$ is not projective. There is an $H$-module $M$ such that $E\oplus M$ is a tilting object of $\mathcal{H}$ and $M$ is a tilting object of $E^\perp$.
In particular, $E^\perp$ is a hereditary abelian category with tilting object $M$ and the result follows from Lemma \ref{l:prep-cat}.

\end{proof}

\section{Proofs of the main results}\label{s:proofs}

\subsection{Proof of Theorem \ref{t:main-thm-1}}

$(1)\Rightarrow (2)\Rightarrow (3)$ are obviously.

$(3)\Rightarrow (1)$: Since $\der^b(\mathcal{H})$ admits silting objects, the Grothendieck group $\go(\der^b(\mathcal{H}))$ is a free abelian group of finite rank. We prove the statement by induction on $n:=\opname{rank} \go(\der^b(\mathcal{H}))$. This is clearly for the case $n=1$.
Now assume that this is true for $n<n_0$. Let $\mathcal{H}$ be a hereditary abelian category with  $\opname{rank}\go(\der^b(\mathcal{H}))=n_0$.
 Let $T=E\oplus \overline{T}$ be a basic silting object of $\der^b(\mathcal{H})$, where $E$ is an indecomposable direct summand. Without loss of generality, we may assume that $E\in \mathcal{H}$. In particular, $E$ is an exceptional object. 
 By Lemma \ref{l:prep-cat}, we have $\der^b(\mathcal{H})/\opname{thick}(E)\cong \der^b(E^\perp)$. Clearly, $E$ is a silting object of $\opname{thick}(E)$ and $\opname{thick}(E)$ is functorially finite in $\der^b(\mathcal{H})$. Denote by $\mathbb{L}:\der^b(\mathcal{H})\to \der^b(\mathcal{H})/\opname{thick}(E)$ the localization functor.
 By Theorem \ref{t:silting-reduction}, $\mathbb{L}(\overline{T})$ is a silting object of $\der^b(\mathcal{H})/\opname{thick}(E)\cong\der^b(E^\perp)$. Consequently, $\opname{rank} \go(\der^b(E^\perp))=n_0-1$. By induction, $E^\perp$ has a tilting object. We conclude that  $\mathcal{H}$ has a tilting object by Lemma \ref{l:silting-tilting}.
 
 $(1)\Rightarrow (4)$: Let $T$ be a basic tilting object of $\mathcal{H}$. Denote by $A=\End_\mathcal{H}(T)$ the endomorphism algebra of $T$. Let $\mod A$ be the category of finitely generated right $A$-modules.  Denote by $S_1,\ldots, S_n$ the pairwise non-isomorphic simple $A$-modules.
 We have an equivalence of triangulated categories $\mathbb{F}: \der^b(\mod A)\to \der^b(\mathcal{H})$.
 Clearly, $\mathbb{F}(S_1),\ldots, \mathbb{F}(S_n)$ is a SMC of $\der^b(\mathcal{H})$. Moreover, its $\opname{Ext}$-quiver is acyclic by Lemma~\ref{l:no-cycle}.
 
 $(4)\Rightarrow (1)$:  Since $\der^b(\mathcal{H})$ admits a SMC, the Grothendieck group $\go(\der^b(\mathcal{H}))$ is a free abelian group of finite rank. We prove the statement by induction on $n:=\opname{rank} \go(\der^b(\mathcal{H}))$.
 This is clearly true for $n=1$. Assume that this is true for $n<n_0$. Let $\mathcal{H}$ be a hereditary abelian category with $\opname{rank} \go(\der^b(\mathcal{H}))=n_0$ and $\mathcal{X}=\{X_1,\ldots, X_{n_0}\}$ a SMC of $\der^b(\mathcal{H})$ with acyclic $\opname{Ext}$-quiver.  Since the $\opname{Ext}$-quiver of $\mathcal{X}$ is acyclic, we deduce that  $X_1,\ldots, X_{n_0}$ are exceptional. Furthermore, we may renumerate $X_i$ to assume that 
 \[\Hom_{\der^b(\mathcal{H})}(X_i, X_j[1])=0 ~\text{whenever $i>j$.}
 \]
 Without loss of generality, we may assume that $X_1\in \mathcal{H}$.
Denote by $\mathcal{H}_{X_{1}}$ the smallest extension-closed subcategory of $\der^b(\mathcal{H})$ containing $X_1$. Since $X_1$ is exceptional, we have $\mathcal{H}_{X_1}=\add X_1$. Consequently, $\mathcal{H}_{X_1}$ satisfies the conditions of Theorem \ref{t:smc-reduction}. Let $\mathcal{Z}:=X_1[\geq 0]^\perp\cap \!^\perp X_1[\leq 0]$. Since $\mathcal{X}$ is a SMC, we deduce that $X_2,\ldots, X_{n_0}\in \mathcal{Z}$.  
By Theorem \ref{t:smc-reduction}, $\{X_2, \ldots, X_{n_0}\}$ is a SMC of $\mathcal{Z}\cong \der^b(\mathcal{H})/\opname{thick}(X_1)$. According to Lemma \ref{l:prep-cat}, $\der^b(\mathcal{H})/\opname{thick}(X_1)\cong \der^b(X_1^{\perp})$. Hence $\opname{rank} \go(\der^b(X_1^\perp))=n_0-1$.

Recall that $\langle 1\rangle$ is the suspension functor of $\mathcal{Z}$. For every $X_i$, $i=2, \ldots, n_0$, we have a triangle in $\der^b(\mathcal{H})$
\[R_{X_i}\xrightarrow{f_i} X_i[1]\to X_i\langle 1\rangle \to R_{X_i}[1],
\]
where $f_i$ is a minimal right $\add X_1$-approximation of $X_i[1]$. Applying $\Hom_{\der^b(\mathcal{H})}(X_j, -)$ to the above triangle, we obtain 
\[\cdots\to \Hom_{\der^b(\mathcal{H})}(X_j,X_i[1])\to \Hom_{\der^b(\mathcal{H})}(X_j,X_i\langle 1\rangle)\to \Hom_{\der^b(\mathcal{H})}(X_j,R_{X_i}[1])\to \cdots.
\]
Consequently, for any  $1<i\leq j$, we obtain $\Hom_{\mathcal{Z}}(X_j, X_{i}\langle 1\rangle)=\Hom_{\der^b(\mathcal{}H)}(X_j, X_i\langle 1\rangle)=0$. In particular, the $\opname{Ext}$-quiver of the SMC $\{X_2, \ldots, X_{n_0}\}$ of $\mathcal{Z}\cong \der^b(\mathcal{H})/\opname{thick}(X_1)\cong \der^b(X_1^\perp)$ is acyclic.
 By induction, $X_1^\perp$ has a tilting object. We conclude that $\mathcal{H}$ has a tilting object by Lemma \ref{l:silting-tilting}.

\subsection{Proof of Theorem \ref{t:presilting-partial}}

We prove this statement by induction on the rank $n:=\opname{rank}\go(\der^b(\mathcal{H}))$ of the Grothendieck group $\go(\der^b(\mathcal{H}))$. This is clear for $n=1$. Assume that this is true for $n<n_0$. Let $\mathcal{H}$ be a hereditary abelian category with tilting objects such that $\opname{rank}\go(\der^b(\mathcal{H}))=n_0$. Let $T=T_1\oplus T_2\oplus \cdots\oplus T_r$ be a basic presilting object of $\der^b(\mathcal{H})$ with indecomposable direct summands $T_1,\ldots, T_r$. We may assume that $T_i\in \mathcal{H}[t_i]$ such that 
\begin{enumerate}
\item[$\circ$] $t_1\leq t_2\leq \cdots\leq t_r=0$;
\item[$\circ$] $\Hom_{\der^b(\mathcal{H})}(T_i,T_j)=0$ whenever $i>j$ by Lemma~\ref{l:no-cycle}.
\end{enumerate}
Let us rewrite $\overline{T}:=T_1\oplus \cdots\oplus T_{r-1}$. We clearly have $\Hom_{\der^b(\mathcal{H})}(T_r[i], \overline{T})=0$ for all $i\in \Z$. In particular, $\overline{T}\in \opname{thick}(T_r)^\perp$. Recall that we have an equivalence $\opname{thick}(T_r)^\perp\cong \der^b(\mathcal{H})/\opname{thick}(T_r)$ by Lemma \ref{l:prep-cat}. As a consequence, $\overline{T}$ is a presilting of $\der^b(\mathcal{H})/\opname{thick}(T_r)$. On the other hand, by Theorem~\ref{t:verdier-quotient}, $\der^b(\mathcal{H})/\opname{thick}(T_r)$ is triangle equivalent to $\der^b(\mathcal{H}')$ for a hereditary abelian category $\mathcal{H}'$ with tilting objects such that $\opname{rank}\go(\der^b(\mathcal{H}'))=n_0-1$. By induction, every presilting object of $\der^b(\mathcal{H}')$ is a partial silting object. Consequently, there is an object $N\in \opname{thick}(T_r)^\perp$ such that $\overline{T}\oplus N$ is a silting object of $\opname{thick}(T_r)^\perp\cong \der^b(\mathcal{H})/\opname{thick}(T_r)$. By Theorem \ref{t:silting-reduction}, there is a silting object $M$ of $\mathcal{T}$ such that $T_r\in \add M$ and $\mathbb{L}(M)=\mathbb{L}(\overline{T}\oplus N)$, where $\mathbb{L}:\der^b(\mathcal{H})\to \der^b(\mathcal{H})/\opname{thick}(T_r)$ is the localization functor.  It remains to show that $\overline{T}\in \add M$.

Denote by $\mathcal{S}_{T_r}^{\leq 0}=\cup_{l\geq 0}\add T_r*\add T_r[1]*\cdots *\add T_r[l]$ and $\mathcal{S}_{T_r}^{<0}=\mathcal{S}_{T_r}^{\leq 0}[1]$. Since $T=\overline{T}\oplus T_r$ is a presilting object, we conclude that $\Hom_{\mathcal{T}}(\overline{T},X)=0$ for any $X\in \mathcal{S}_{T_r}^{<0}$. By the construction of $M$ (cf. Section \ref{s:silting-theory}), we conclude that $\overline{T}$ is a direct summand of $M$. In particular, $T$ is a partial silting object.

\subsection{Proof of Theorem \ref{t:main-thm-3}}
Let $T$ be a basic tilting object of $\mathcal{H}$. Denote by $\Lambda=\End_\mathcal{H}(T)$ the endomorphism algebra of $T$, which has global dimension at most $2$. Let $\mod \Lambda$ be the category of finitely generated right $\Lambda$-modules and $\der^b(\mod \Lambda)$ its bounded derived category. It is well-known that $\der^b(\mod \Lambda)\cong \der^b(\mathcal{H})$.  By \cite[Theorem 6.1]{KY}, there is a one-to-one correspondence between  $\opname{silt}\der^b(\mathcal{H})$ and  $\opname{SMC} \der^b(\mathcal{H})$.
Let $P$ be a basic silting object of $\der^b(\mathcal{H})$ and $A:=\End_{\der^b(\mathcal{H})}(P)$ the endomorphism algebra of $P$. Denote by $\mod A$ the category of finitely generated right $A$-modules.
Agian by \cite[Theorem 6.1]{KY},  $P$ determines a bounded $t$-structure of $\der^b(\mathcal{H})$ whose heart is equivalent to $\mod A$ (cf. also \cite[Section 2.4]{Fu}). Via this equivalence, the simple $A$-modules  form a SMC of $\der^b(\mathcal{H})$ corresponding to $P$.  According to Lemma \ref{l:no-cycle}, we conclude that the $\Ext$-quiver of an arbitrary SMC of $\der^b(\mathcal{H})$ is acyclic.  In particular, this implies the `only if' part.

It remains to prove the `if' part and we prove it by induction on the rank $n:=\opname{rank}\go(\der^b(H))$ of the Grothendieck group $\go(\der^b(\mathcal{H}))$. This is clear for $n=1$. Assume that this is true for $n<n_0$. Let $\mathcal{H}$ be a hereditary abelian category  with tilting objects such that $\opname{rank}\go(\der^b(\mathcal{H}))=n_0$. Let $\mathcal{X}=\{X_1,\ldots, X_r\}$ be a pre-SMC of $\der^b(\mathcal{H})$ whose $\Ext$-quiver is acyclic. Without loss of generality, we may assume that $X_1\in \mathcal{H}$ and
\[\Hom_{\der^b(\mathcal{H})}(X_i, X_j[1])=0 \ \text{whenever $i\geq j$.}
\]

Denote by $\mathcal{R}:=\{X_1\}$, which is a pre-SMC of $\der^b(\mathcal{H})$. Denote by $\mathcal{H}_\mathcal{R}$ be the smallest extension-closed subcategory of $\der^b(\mathcal{H})$ containing $\mathcal{R}$. Since $X_1$ is exceptional, we have $\mathcal{H}_\mathcal{R}=\add X_1$. Consequently, $\mathcal{H}_\mathcal{R}$ satisfies  the conditions of Theorem \ref{t:smc-reduction} and there is a bijection between $\opname{SCM}_\mathcal{R}\der^b(\mathcal{H})$ and $\opname{SCM} \der^b(\mathcal{H})/\opname{thick}(X_1)$.
Denote by $\mathcal{Z}:=\mathcal{R}[\geq 0]^\perp\cap ^\perp\mathcal{R}[\leq 0]$. By definition of pre-SMC, we clearly have $X_2,\ldots, X_r\in \mathcal{Z}$. Recall that $\langle 1\rangle$ is the suspension functor of $\mathcal{Z}$. By definition of $\langle 1\rangle$ (cf. Section \ref{ss:smc}), for each object $X\in \mathcal{Z}$ and $n>0$, we have $X\langle n\rangle\in X[n]*\add X_1[n]*\cdots*\add X_1[1]$. As a consequence, we obtain $\Hom_\mathcal{Z}(X_i\langle n\rangle, X_j)=0$ for $2\leq i,j\leq r$ and $n>0$. In particular, $\{X_2,\ldots, X_r\}$ is a pre-SMC of $\mathcal{Z}$.
Let  $\mathbb{L}:\der^b(\mathcal{H})\to \der^b(\mathcal{H})/\opname{thick}(X_1)$ the localization functor. It follows from Theorem \ref{t:smc-reduction} (1) that $\{\mathbb{L}(X_2),\ldots, \mathbb{L}(X_r)\}$ is a pre-SMC of $\der^b(\mathcal{H})/\opname{thick}(X_1)$. Similar to the proof of $(4)\Rightarrow (1)$ of Theorem \ref{t:main-thm-1},  one can show that the $\Ext$-quiver of $\{\mathbb{L}(X_2),\ldots, \mathbb{L}(X_r)\}$ is acyclic. According to Theorem \ref{t:verdier-quotient}, we have $\der^b(\mathcal{H})/\opname{thick}(X_1)\cong \der^b(\mathcal{H}')$, where $\mathcal{H}'$ is a hereditary abelian category with tilting objects such that $\opname{rank}\go(\der^b(\mathcal{H}'))=n_0-1$. By induction, $\{\mathbb{L}(X_2),\ldots, \mathbb{L}(X_r)\}$ can be completed into a SMC of $\der^b(\mathcal{H})/\opname{thick}(X_1)$. Now the result follows from Theorem \ref{t:smc-reduction} (2).

 %\bibliographystyle{amsplain}
%\bibliography{stanKeller}

\def\cprime{$'$} \def\cprime{$'$}
\providecommand{\bysame}{\leavevmode\hbox to3em{\hrulefill}\thinspace}
\providecommand{\MR}{\relax\ifhmode\unskip\space\fi MR }
\providecommand{\MRhref}[2]{%
  \href{http://www.ams.org/mathscinet-getitem?mr=#1}{#2}
}
\providecommand{\href}[2]{#2}
 
\end{document}